\documentclass[12pt]{amsart}

\usepackage{amsmath,amssymb,microtype,graphicx}
\usepackage[T1]{fontenc}
\usepackage[utf8]{inputenc}
\usepackage[english]{babel}
\usepackage[top=3.0cm,bottom=3.0cm,left=3.2cm,right=3.2cm]{geometry}
\usepackage[bookmarksdepth=2,linktoc=page,colorlinks,linkcolor={red!80!black},citecolor={red!80!black},urlcolor={blue!80!black}]{hyperref}
\newcommand{\arxiv}[1]{\href{http://arxiv.org/pdf/#1}{arXiv:#1}}   
\hypersetup{bookmarksdepth=1}  
\usepackage[mathcal]{euscript}                               
\usepackage{mathptmx}                                        
\usepackage[backgroundcolor=yellow,linecolor=yellow,textsize=footnotesize]{todonotes} \makeatletter \providecommand \@dotsep{5} \def\listtodoname{List of Todos} \def\listoftodos{\@starttoc{tdo}\listtodoname}\makeatother 

\usepackage{etoolbox}
\makeatletter
\patchcmd{\@startsection}{\@afterindenttrue}{\@afterindentfalse}{}{}             
\patchcmd{\section}{\scshape}{\bfseries}{}{}\renewcommand{\@secnumfont}{\bfseries} 
\patchcmd{\@settitle}{\uppercasenonmath\@title}{\large}{}{}
\patchcmd{\@setauthors}{\MakeUppercase}{}{}{}
\addto{\captionsenglish}{} 
\makeatother

\usepackage{fancyhdr}

\theoremstyle{plain}
\newtheorem{thm}{Theorem}
\newtheorem{prop}[thm]{Proposition}
\newtheorem*{question*}{Question}

\DeclareSymbolFont{cmlargesymbols}{OMX}{cmex}{m}{n}                           
\DeclareMathSymbol{\mycoprod}{\mathop}{cmlargesymbols}{"60}\let\coprod\mycoprod

\renewcommand\emph[1]{\textbf{#1}}         

\DeclareMathOperator{\Star}{Star}
\DeclareMathOperator{\csm}{csm}

\newcommand\R{{\mathbb R}}
\newcommand\cF{{\mathcal F}}
\newcommand\fc{{\mathfrak c}}
\newcommand\rk{\textup{rk\,}}
\renewcommand\leq{\leqslant}
\renewcommand\geq{\geqslant}
\newcommand\cover{<:}

  \title{Higher balancing for locally tropically convex tropical varieties}
 
\author{Emilio Assemany}
\address{\rm Emilio Assemany, Instituto Nacional de Matem\'atica Pura e Aplicada, Rio de Janeiro, Brazil}
\email{\href{mailto:emiliopassemany@gmail.com}{emiliopassemany@gmail.com}}

\author{Oliver Lorscheid}
\address{\rm Oliver Lorscheid, Instituto Nacional de Matem\'atica Pura e Aplicada, Rio de Janeiro, Brazil}
\email{\href{mailto:oliver@impa.br}{oliver@impa.br}}

\begin{document}

\begin{abstract} 
 In this text, we show that locally tropically convex tropical varieties in $\R^n$ have locally a canonical polyhedral structure that satisfies higher balancing conditions at all cells of positive codimension.
\end{abstract}

\maketitle


\section*{Introduction}
\label{introduction}

\subsection*{Locally tropically convex tropical varieties}

A subset $X$ in $\R^n$ is \emph{tropically convex} if for all $x,y\in X$ and $a,b\in\R$, the tropical linear combination $z=(a\odot x)\oplus(b\odot y)$ (with coordinates $z_i=\min\{a+x_i,b+y_i\}$) is contained in $X$. A subset $X$ in $\R^n$ is \emph{locally tropically convex} if every point of $X$ has an open tropically convex neighborhood.

By \cite[Prop.\ 3.3]{Hampe15}, a tropical variety $X$ is locally tropically convex if and only if for every point $p$ of $X$, the (underlying set of the) star $\Star_X(p)$ at $p$ is the (underlying set of the) Bergman fan of a matroid. Therefore the balancing conditions of locally tropically convex tropical varieties can be traced back to the balancing conditions for Bergman fans.

Before we turn to Bergman fans in more detail, we cite some general facts about locally tropically convex tropical varieties; for definitions cf.\ \cite{Hampe15} and \cite{Speyer08}. A tropical variety is locally tropically convex if and only if it is tropically convex as a set (\cite[Thm.\ 1.2]{Hampe15}), and tropical linear spaces are tropically convex (\cite[Prop.\ 2.14]{Hampe15}). In fact, the only discrepancy between tropical linear spaces and tropically convex tropical varieties lies in the possibility of higher weights: by \cite[Thm.\ 1.1]{Hampe15}, the underlying set of a tropically convex tropical variety is equal to the underlying set of a tropical linear space.


\subsection*{A preview on higher balancing}

It is well-known that the Bergman fan of a matroid is a tropical variety, which means that it satisfies the balancing condition for all top-dimensional cones containing a given cone of codimension $1$ (with respect to a constant weight function). In this text, we show that a Bergman fan satisfies balancing conditions for polyhedra of any codimension. These higher balancing conditions require a distinction of the cones of a Bergman fan according to their `types', which is a finer invariant than the dimension. The formulation of the higher balancing conditions requires some preparatory definitions and can be found in Theorem \ref{thm: higher balancing}.


\subsection*{A remark on the relevance of higher balancing}

Locally tropically convex tropical varieties are blessed with the property that they look locally like a Bergman fan, and therefore inherit a canonical polyhedral structure that satisfies higher balancing conditions with respect to a constant weight function. Other types of tropical varieties do not come with such an intrinsic structure, but one needs additional information in order to extend the weight function to polyhedra of higher codimension.

Tentative calculations show that in good cases a tropical basis for the tropical variety provides enough structure to define weight functions for which higher balancing holds. In particular, this works well for hypersurfaces with one defining equation. We have hopes to extend this to all tropicalizations of classical varieties. It might also apply to tropical prevarieties that are defined by tropical ideals in the sense of \cite{Maclagan-Rincon18}, but at the time of writing it is not even clear whether they are balanced in codimension $1$.

To conclude, we see higher balancing as an indication for that there might be interesting information about the tropicalizations of classical varieties that has not been used so far. Our hope is that this additional information finds a satisfactory explanation in terms of tropical scheme theory, as developed in \cite{Giansiracusa-Giansiracusa16}, \cite{Lorscheid15} and \cite{Maclagan-Rincon14}. In particular, we would like to propose the following question as a guiding problem for further developments in tropical scheme theory.

\begin{question*}
 For which subschemes of the tropical torus (i.e.\ ideals in the semiring of tropical Laurent polynomials) can we make sense of higher balancing?
\end{question*}


\subsection*{Acknowledgements}

We would like to thank Felipe Rinc\'on and Kristin Shaw for their comments on a first version on this paper.


\section{The Bergman fan of a matroid}
\label{section: the Bergman fan of a matroid}

The Bergman complex of a matroid was introduced by Sturmfels in \cite{Sturmfels02}. Subsequently the related notion of the Bergman fan of a matroid was introduced by Ardila and Klivans in \cite{Ardila-Klivans06}. We will review this theory in the following. For a pleasant introduction, cf.\ section 2.2 of \cite{Huh14}. 

Let $M$ be a matroid with ground set $E$. The \emph{Bergman fan of $M$} is the fan $B(M)$ in $\R^E$ whose cones are defined as follows. Let $\{e_i\}_{i\in E}$ be the standard basis of $\R^E$. For a subset $F$ of $E$, we define $e_F=\sum_{i\in F}e_i$. A \emph{flag of flats} is a tuple $\cF=(F_0,\dotsc,F_d)$ of flats $F_i$ of $M$ such that
\[
 \emptyset \, = \, F_0 \ \subsetneq \ F_1 \ \subsetneq \ \dotsb \ \subsetneq \ F_d \, = \, E.
\]
The \emph{cone of $\cF$} is defined as
\[
 \fc_\cF \ = \ \sum_{i=0}^{d-1} \R^+ \cdot e_{F_i} + \R\cdot e_E \ = \ \bigg\{ \, \sum_{i=0}^{d} \lambda_ie_{F_i} \, \bigg| \, \lambda_i\in\R^+\text{ for }i=1,\dotsc,d-1\text{ and }\lambda_d\in\R \, \bigg\}
\]
where $\R^+$ are the nonnegative reals. 

By definition, $\fc_\cF$ is a rational cone of dimension $d$ (in the sense of toric geometry, cf.\ \cite{Fulton93}). The maximal linear subspace contained in $\fc_\cF$ is the line spanned by $e_E$. Let $|\cF|=\{F_0,\dotsc,F_d\}$. We have an inclusion $\fc_{\cF'} \subset \fc_{\cF}$ if and only if $|\cF'|\subset|\cF|$. In this case $\fc_{\cF'}$ is a face of $\fc_\cF$, and every face of $\fc_\cF$ is of the form $\fc_{\cF'}$ for some flag of flats $\cF'$. In particular, every cone $\fc_\cF$ contains $\fc_{(\emptyset,E)}=\R\cdot e_E$ as its unique $1$-dimensional face. Given two flags of flats $\cF$ and $\cF'$, the intersection of the associated cones is $\fc_\cF\cap\fc_{\cF'}=\fc_{\cF''}$ where $\cF''$ is the flag of flats with $|\cF''|=|\cF|\cap|\cF'|$. This shows that $B(M)$ is an equidimensional polyhedral complex whose dimension is equal to the rank $\rk M$ of $M$. 

Note that some authors (e.g.\ Hampe in \cite{Hampe15}) define the Bergman fan of $M$ as the image of $B(M)$ in $\R^E/\R\cdot e_E$. This image is indeed a fan in the sense of toric geometry since all cones become strictly convex. Concerning balancing conditions, it does not make any essential difference, which version of the Bergman fan one uses. For our purposes, we find it more convenient to follow the definition of this paper.

Note further that the matroid $M$ is determined by its Bergman fan $B(M)$ since a subset $F$ of $E$ is a flat if and only if $e_F$ is contained in the \emph{underlying set of $B(M)$}, which is the union $|B(M)|=\bigcup\fc_\cF$ of all cones $\fc_\cF$ of $B(M)$.

Let $\cF=(F_0,\dotsc,F_d)$ be a flag of flats. The \emph{type of $\cF$} is the tuple $(\rk F_0,\dotsc,\rk F_d)$, which is a tuple of strictly increasing integers with $\rk F_0=0$ and $\rk F_d=\rk M$.


\section{Balancing in codimension \texorpdfstring{$1$}{1}}

The Bergman fan of a matroid is a tropical variety in the sense that it is balanced at cones of codimension $1$ with respect to a constant weight function on the top-dimensional cones. This balancing condition was first considered in Speyer's thesis \cite{Speyer05}, and can be formulated as follows for Bergman fans. For details on matroid theory, we refer to \cite{Oxley92}.

Let $M$ be a matroid with ground set $E$. Given flats $F$ and $F'$, we write $F\leq F'$ if $F\subset F'$. We say that \emph{$F'$ covers $F$}, and write $F\cover F'$, if $F\leq F'$ and $\rk F'=\rk F+1$, i.e.\ if there exists no flat strictly in between $F$ and $F'$. 

\begin{prop}\label{prop: usual balancing}
 Consider a flag of flats $\cF=(F_0,\dotsc,F_d)$ of length $d=\rk M-1$, i.e.\ the type of $\cF$ is $(0,1,\dotsc,i,i+2,\dotsc,r)$ for some $i\in\{0,\dotsc, r-2\}$ where $r=\rk M$. Then 
 \[
  \sum_{F_{i}\cover F'\leq F_{i+1}} e_{F'} \quad \in \quad \fc_\cF. 
 \]
\end{prop}

\begin{proof}
 In the following, we reproduce the short argument from Huh's thesis (\cite[Prop.\ 16]{Huh14}).

 As a first step, we consider the restriction $M'=M\vert_{F_{i+1}}$ of $M$ to $F_{i+1}$, which results from $M$ by deleting $E-F_{i+1}$. The ground set of $M'$ is $E'=F_{i+1}$ and its flats are precisely those flats of $M$ that are contained in $F_{i+1}$. The matroid axiom for flats, applied to $F_{i}$ as a flat of $M'$, states that
 \[
  F_{i+1}\,-\,F_{i} \ = \ \coprod_{F_{i}\cover F'\leq F_{i+1}} F' \,-\, F_{i}.
 \]
 Therefore
 \[
  e_{F_{i+1}}-e_{F_{i}} \ = \ \sum _{F_{i}\cover F'\leq F_{i+1}} (e_{F'} - e_{F_{i}}) 
 \] 
 and thus
 \[ 
  \sum _{F_{i}\cover F'\leq F_{i+1}} e_{F'} \ = \ e_{F_{i+1}}+(m-1)e_{F_{i}} \quad \in \quad \fc_{\cF_{i}}
 \]
 where $m\geq1$ is the number of flats $F'$ with $F_{i}\cover F'\leq F_{i+1}$.
\end{proof}


\section{Higher balancing}
\label{section: higher balancing}

We are prepared to state and prove the main result of this text.

\begin{thm}\label{thm: higher balancing}
 Let $M$ be a matroid with ground set $E$ and $\cF=(F_0,\dotsc,F_d)$ a flag of flats of type $(r_0,\dotsc,r_d)$. Let $i$ and $k$ be integers such that $0\leq i\leq d-1$ and $1\leq k\leq r_{i+1}-r_{i}$. Then
 \begin{multline*}
  \sum_{F^{(0)}\cover \dotsb\cover F^{(k)}\leq F_{i+1}} \hspace{-20pt} (e_{F^{(k)}}-e_{F^{(k-1)}}) \quad + \quad \sum_{l=1}^{k-1} (-1)^{k-l}  \hspace{-25pt} \sum_{F^{(0)}\cover \dotsb\cover F^{(l)}\leq F_{i+1}} \hspace{-20pt} (e_{F_{i+1}}-e_{F^{(l-1)}})  \\
  + \quad (-1)^k (e_{F_{i+1}}-e_{F_{i}}) \quad = \quad 0
 \end{multline*}
 where $F^{(0)}=F_{i}$ is fixed and $F^{(1)},\dotsc,F^{(k)}$ vary over all possible flats. We call this relation the \emph{$(i,k)$-balancing condition at $\fc_\cF$}.
\end{thm}

\begin{proof}
 We prove the result by induction on $k$. For the sake of presentation, we will show that the first sum in the balancing condition equals the inverse of the other terms. 
 
 The case $k=1$ follows by the same argument that we have used to prove Proposition \ref{prop: usual balancing} where we note that the proof did not make any use of the assumptions that $\cF$ is of length $d=\rk M-1$ and that $r_{i+1}-r_i=2$. Therefore, we derive that
 \[
  \sum_{F^{(0)}\cover F^{(1)}\leq F_{i+1}} \hspace{-20pt} (e_{F^{(1)}}-e_{F^{(0)}}) \quad = \quad e_{F_{i+1}}-e_{F_{i}} \quad = \quad -\Big[ (-1)^k(e_{F_{i+1}}-e_{F_{i}})\Big],
 \]
 where we use that $(-1)^k=-1$ and that there is no middle term of the form ``$\sum_{l=1}^{k-1}\dotsc$'' since $k-1<1$.
 
 If $k>1$, then we can split the sequences $F^{(0)}\cover \dotsb\cover F^{(k)}\leq F_{i+1}$ into $F^{(0)}\cover F^{(1)}$ and $F^{(1)}\cover \dotsb\cover F^{(k)}\leq F_{i+1}$, which yields an equality
 \[
  \sum_{F^{(0)}\cover \dotsb\cover F^{(k)}\leq F_{i+1}} \hspace{-20pt} (e_{F^{(k)}}-e_{F^{(k-1)}}) \quad = \quad \sum_{F^{(0)}\cover F^{(1)}\leq F_{i+1}} \Bigg[ \sum_{F^{(1)}\cover \dotsb\cover F^{(k)}\leq F_{i+1}} \hspace{-20pt} (e_{F^{(k)}}-e_{F^{(k-1)}}) \Bigg].
 \]
 Applying the inductive hypothesis to the sum inside the brackets transforms this expression into
 \[
  -\sum_{F^{(0)}\cover F^{(1)}\leq F_{i+1}} \Bigg[ \sum_{l=1}^{k-2} (-1)^{(k-1)-l}  \hspace{-25pt} \sum_{F^{(1)}\cover \dotsb\cover F^{(l+1)}\leq F_{i+1}}\hspace{-20pt} (e_{F_{i+1}}-e_{F^{(l)}}) + (-1)^{k-1} (e_{F_{i+1}}-e_{F^{(1)}})\Bigg].
 \]
 Merging the outer sum over $F^{(0)}\cover F^{(1)}\leq F_{i+1}$ with the inner terms and replacing $l$ by $l-1$ yields
 \[
  -\Bigg[ \sum_{l=2}^{k-1} (-1)^{k-l} \hspace{-25pt} \sum_{F^{(0)}\cover \dotsb\cover F^{(l)}\leq F_{i+1}}\hspace{-20pt} (e_{F_{i+1}}-e_{F^{(l-1)}}) + (-1)^{k-1} \hspace{-18pt}\sum_{F^{(0)}\cover F^{(1)}\leq F_{i+1}} \hspace{-13pt}(e_{F_{i+1}}-e_{F^{(1)}})\Bigg].
 \]
 We can apply the case $k=1$ to the sum on the right hand side and get
 \[
  (-1)^{k-1} \hspace{-18pt}\sum_{F^{(0)}\cover F^{(1)}\leq F_{i+1}} \hspace{-13pt}(e_{F_{i+1}}-e_{F^{(1)}}) \quad = \quad (-1)^{k-1} \hspace{-18pt}\sum_{F^{(0)}\cover F^{(1)}\leq F_{i+1}} \hspace{-13pt}(e_{F_{i+1}}-e_{F^{(0)}}) + (-1)^k (e_{F_{i+1}}-e_{F_{i}}).
 \]
 Substituting this term in the expression above produces the desired outcome
 \[
  -\Bigg[ \sum_{l=1}^{k-1} (-1)^{k-l} \hspace{-25pt} \sum_{F^{(0)}\cover \dotsb\cover F^{(l)}\leq F_{i+1}}\hspace{-20pt} (e_{F_{i+1}}-e_{F^{(l-1)}}) + (-1)^{k} (e_{F_{i+1}}-e_{F_{i}})\Bigg]. \qedhere
 \]
\end{proof}


\section{Geometric interpretation of higher balancing}
\label{section: geometric interpretation of higher balancing}

Let $\cF=(F_0,\dotsc,F_d)$ be a flag of flats in $M$ of type $(r_0,\dotsc,r_d)$. The cases of $(i,k)$-balancing for $k=r_{i+1}-r_i$ are degenerate and we will discuss them below. For now we assume that $k<r_{i+1}-r_i$ and write $F<F'$ for $F\subsetneq F'$.

To begin with, we observe that the flags $F^{(0)}\cover \dotsb\cover F^{(l)}< F_{i+1}$ that occur as indices of the sums in the $(i,k)$-balancing condition can be identified with the cones $\fc_{\cF'}$ with 
\[
 \cF'=(F_0,\dotsc,F_{i},F^{(1)},\dotsc,F^{(l)},F_{i+1},\dotsc,F_d).
\]
Thus we can interpret the sum as varying over all cones in $B(M)$ of type 
\[
 (r_0,\dotsc,r_{i},r_{i}+1,\dotsc,r_{i}+l,r_{i+1},\dotsc,r_d)
\]
that contain $\fc_\cF$.
 
Therefore the $(i,1)$-balancing condition
\[
 \sum_{F_{i}\cover F^{(1)}\leq F_{i+1}} \hspace{-20pt} (e_{F^{(1)}}-e_{F_{i}}) - (e_{F_{i+1}}-e_{F_{i}}) \quad = \quad 0
\]
can be rewritten as
\[
 \sum_{\substack{\cF'=(F_0',\dotsc,F_{d+1}')\text{ of type }\\ (r_0,\dotsc,r_{i},r_{i}+1,r_{i+1},\dotsc,r_d)\\\text{such that }\fc_\cF\subset\fc_{\cF'}}} \hspace{-20pt} e_{F_{i+1}'} \quad = \quad e_{F_{i+1}} + (m-1) e_{F_{i}} \quad \in \quad \fc_\cF
\]
where $m\geq1$ is the number of flags $\cF'$ of type $(r_0,\dotsc,r_{i},r_{i}+1,r_{i+1},\dotsc,r_d)$ such that $\fc_\cF\subset\fc_{\cF'}$. Note that $e_{\cF'}$ is a primitive vector for $\fc_{\cF'}$ modulo $\fc_\cF$, which is a ray. In particular, this recovers the usual balancing condition for tropical varieties in the case that $\fc_\cF$ is of dimension $\rk M-1$.

For $k>1$, a geometric interpretation of $(i,k)$-balancing involves different `types' of `primitive vectors', one for each ray of $\fc_{\cF'}$ that is not contained in $\fc_{\cF}$. Without spelling out the obvious formula, the $(i,k)$-balancing condition states that a certain linear combination of primitive vectors of (the rays of) cones containing $\fc_\cF$, ordered by their types, is contained in the linear subspace spanned by $\fc_\cF$.

Balancing in the degenerate case $k=r_{i+1}-r_i=1$ yields the trivial relation
\[
 (e_{F_{i+1}}-e_F)-(e_{F_{i+1}}-e_F) \ = \ 0.
\]
If $k=r_{i+1}-r_i>1$, then $(i,k)$-balancing results from $(i,k-1)$-balancing after a trivial rearrangement of terms, which brings, however, the relation into a more symmetric shape. In particular, $(i,k)$-balancing implies that
\[
 \sum_{l=2}^k \quad (-1)^k \hspace{-10pt} \sum_{F_i\cover F^{(1)}\cover \dotsb\cover F^{(l)}\leq F_{i+1}} \hspace{-10pt} e_{F^{(l-1)}} 
\]
is contained in the linear subspace spanned by $\fc_{\cF}$.


\section{Relation to CSM-balancing}
\label{section: relation to CSM-balancing}

Lopez de Medrano, Rinc\'on and Shaw introduce in \cite{Lopez-Rincon-Shaw17} the \textit{$k$-th Chern-Schwartz-MacPherson cycle $\csm_k(M)$} of a matroid $M$, which is the $k$-skeleton of the Bergman fan of the matroid where $k$ is an integer between $0$ and $\rk M$. The fan $\csm_k(M)$ comes with certain weights on its top-dimensional cones, which turn $\csm_k(M)$ into a tropical variety, i.e.\ a polyhedral complex that is balanced in codimension $1$. We refer to this result by \textit{CSM-balancing} for short.

We can reinterpret this result as follows. We can endow all cones $\fc_\cF$ of the Bergman fan $B(M)$ of $M$ with certain weights $\mu_\cF$ such that for every cone $\fc_\cF$ of dimension $k<\rk M$,
\[
 \sum_{\substack{\fc_{\cF}\subset\fc_{\cF'}\\ \dim \fc_{\cF'}=k+1}} \mu_{\cF'} \ e_{\cF'}
\]
lies in the subspace spanned by $\fc_\cF$ where $e_{\cF'}$ is a primitive vector of $\fc_{\cF'}$ modulo $\fc_\cF$.

It is possible to express the weights $\mu_{\cF'}$ as a linear combination of the number of cones containing $\cF'$ (ordered by their types), which indicates a relation to higher balancing in the sense of this paper. We were able to verify that CSM-balancing can be traced back to certain linear combinations of the relations that occur in Theorem \ref{thm: higher balancing} up to codimension $3$, i.e.\ for $\rk M-k-1\leq 3$. We strongly suspect that CSM-balancing can be deduced from higher balancing in general. A proof of this conjecture would be desirable.


\section{Example}
\label{section: example}

As an example, we consider the uniform matroid $M=U_{3,4}$ of rank $3$ with ground set $E=\{1,2,3,4\}$. Its lattice of flats is as follows.
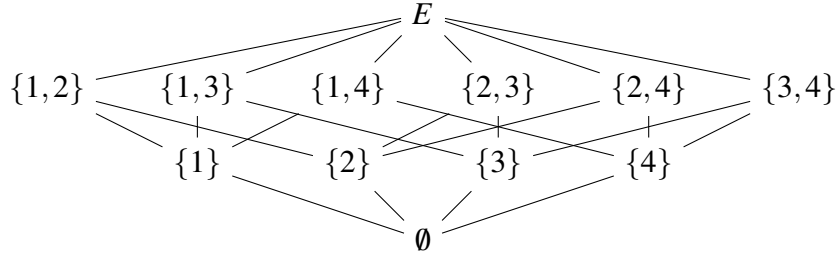
\[
 \begin{tikzpicture}
  \node (E) at (5,3) {$E$};
  \node (12) at (0,2) {$\{1,2\}$};
  \node (13) at (2,2) {$\{1,3\}$};
  \node (14) at (4,2) {$\{1,4\}$};
  \node (23) at (6,2) {$\{2,3\}$};
  \node (24) at (8,2) {$\{2,4\}$};
  \node (34) at (10,2) {$\{3,4\}$};
  \node (1) at (2,1) {$\{1\}$};
  \node (2) at (4,1) {$\{2\}$};
  \node (3) at (6,1) {$\{3\}$};
  \node (4) at (8,1) {$\{4\}$};
  \node (0) at (5,0) {$\emptyset$};
    \path[-] 
    (12) edge (E)
    (13) edge (E)
    (14) edge (E)
    (23) edge (E)
    (24) edge (E)
    (34) edge (E)
    (1) edge (12)
    (1) edge (13)
    (1) edge (14)
    (2) edge (12)
    (2) edge (23)
    (2) edge (24)
    (3) edge (13)
    (3) edge (23)
    (3) edge (34)
    (4) edge (14)
    (4) edge (24)
    (4) edge (34)
    (0) edge (1)
    (0) edge (2)
    (0) edge (3)
    (0) edge (4)
    ;
 \end{tikzpicture}
\]
The different types of cones of the Bergman fan $B(M)$ of $M$ are described in the following table where we fix an identification $E=\{i,j,k,l\}$.

\medskip
\begin{center}
 \begin{tabular}{|l|l|c|c|c|}
  \hline
  type        & typical flag $\cF$            & $(x_1,\dotsc,x_4)$ in $\fc_\cF$ iff.\ & dimension & number of cones \\
  \hline\hline 
  $(0,3)$     & $(\emptyset,E)$               & $x_i=x_j=x_k=x_l$                     & $1$       & $1$\\ 
  \hline
  $(0,1,3)$   & $(\emptyset,\{i\},E)$         & $x_i\geq x_j= x_k= x_l$               & $2$       & $4$ \\
  \hline
  $(0,2,3)$   & $(\emptyset,\{i,j\},E)$       & $x_i= x_j\geq x_k= x_l$               & $2$       & $6$ \\
  \hline
  $(0,1,2,3)$ & $(\emptyset,\{i\},\{i,j\},E)$ & $x_i\geq x_j\geq x_k= x_l$            & $3$       & $12$ \\
  \hline
 \end{tabular}
\end{center}
\medskip

We find non-trivial balancing conditions (i.e.\ $k<r_{i+1}-r_i$) at all cones of positive codimension, i.e.\ at cones of types $(0,1,3)$, $(0,2,3)$ and $(0,3)$. These balancing conditions are, up to permuting $E=\{i,j,k,l\}$, as follows.

\bigskip
\begin{tabular}{lr}
$(1,1)$-balancing at $\fc_{(\emptyset,\{i\},E)}$: & $\displaystyle \sum_{\{i\}\cover \{i,n\}<E} \hspace{-12pt} (e_{\{i,n\}}-e_i) \ - \ (e_E-e_i) \ = \ 0$ \\[25pt]
$(0,1)$-balancing at $\fc_{(\emptyset,\{i,j\},E)}$: & $\displaystyle \sum_{\emptyset\cover \{m\}<\{i,j\}} \hspace{-15pt} (e_{m}-e_\emptyset) \ - \ (e_{\{i,j\}}-e_\emptyset) \ = \ 0$ \\[25pt]
$(0,1)$-balancing at $\fc_{(\emptyset,E)}$: & $\displaystyle \sum_{\emptyset\cover \{m\}<E} \hspace{-10pt} (e_{m}-e_\emptyset) \ - \ (e_E-e_\emptyset) \ = \ 0$ \\[25pt]
$(0,2)$-balancing at $\fc_{(\emptyset,E)}$: & $\displaystyle \hspace{-30pt}\sum_{\emptyset\cover \{m\}\cover \{m,n\}<E} \hspace{-25pt} (e_{\{m,n\}}-e_m) \ - \hspace{-12pt} \sum_{\emptyset\cover \{m\}<E} \hspace{-10pt} (e_{E}-e_\emptyset) \ + \ (e_E-e_\emptyset) \ = \ 0$ \\
\end{tabular}
\bigskip

\noindent
Note that the first two conditions are balancing conditions in codimension $1$, i.e.\ they are the `classical' balancing conditions for tropical varieties. In particular note that these classical balancing conditions are divided into two different types.

\begin{small}
\bibliographystyle{plain}

\end{small}

\end{document}